\documentclass[11pt]{amsart}
\usepackage{amscd,amsmath,amsopn,amssymb,amsthm,multicol}
\usepackage{hyperref}
\DeclareMathOperator{\ad}{ad}

\DeclareMathOperator{\Ad}{Ad}

\DeclareMathOperator{\trace}{trace}

\DeclareMathOperator{\Ric}{Ric}

\renewenvironment{proof}[1][Proof]{\textbf{#1.} }
{\ \rule{0.5em}{0.5em}}
\newtheorem{theorem}{Theorem}
\newtheorem{prop}{Proposition}
\newtheorem{lemma}{Lemma}
\newtheorem{corollary}{Corollary}

\theoremstyle{definition}
\newtheorem{definition}{Definition}
\newtheorem{remark}{Remark}

\begin{document}

\title
[Geodesic orbit manifolds \dots] {Geodesic orbit manifolds and
Killing fields\\ of constant length}

\author{Yu.G.~Nikonorov}

\address{Nikonorov\ Yuri\u\i\  Gennadyevich\newline
South Mathematical Institute of  VSC RAS, \newline
Vladikavkaz,
Russia}
\email{nikonorov2006@mail.ru}

\begin{abstract}
The goal of this paper is to clarify connections between Killing
fields of constant length on a Rimannian geodesic orbit manifold
$(M,g)$ and the structure of its full isometry group. The Lie
algebra of the full isometry group of  $(M,g)$  is identified with
the Lie algebra of Killing fields $\mathfrak{g}$ on $(M,g)$. We
prove the following result: If $\mathfrak{a}$ is an abelian ideal
of $\mathfrak{g}$, then every Killing field $X\in \mathfrak{a}$
has constant length. On the ground of this assertion we give a new proof
of one result of C.~Gordon: Every Riemannian geodesic orbit
manifold of nonpositive Ricci curvature is a symmetric space.

\vspace{2mm} \noindent 2000 Mathematical Subject Classification:
53C20, 53C25, 53C35.

\vspace{2mm} \noindent Key words and phrases: Killing fields of
constant length, homogeneous Riemannian manifolds, geodesic orbit
spaces, symmetric spaces, Ricci curvature.
\end{abstract}

\maketitle

\section{Introduction, notation and useful facts}

All manifolds in this paper are supposed to be connected.
At first, we recall and discuss important definitions.

\begin{definition}\label{GOman}
A Riemannian manifold $(M,g)$ is called a  manifold with
homogeneous geodesics or geodesic orbit manifold (shortly,  {\bf GO-manifold}) if any
geodesic $\gamma $ of $M$ is an orbit of  1-parameter subgroup of
the full isometry group of $(M,g)$.
\end{definition}

\begin{definition}\label{GOsp}
A Riemannian manifold $(M=G/H,g)$, where $H$ is a compact subgroup
of a Lie group $G$ and $g$ is a $G$-invariant Riemannian metric,
is called a space with homogeneous geodesics or geodesic orbit space
(shortly,  {\bf
GO-space}) if any geodesic $\gamma $ of $M$ is an orbit of
1-parameter subgroup of the group $G$.
\end{definition}

This terminology was introduced in
\cite{KV} by O.~Kowalski and L.~Vanhecke, who initiated a systematic study of such spaces.
In the same paper, O.~Kowalski and L.~Vanhecke classified all GO-spaces
of dimension $\leq 6$.
Many  interesting results  about  GO-manifolds
and its subclasses one can find  in \cite{AA, AN, AV, BerNik,
BerNik3, BerNikClif, DuKoNi,Ta, Tam, W1}, and in the references
therein.
In \cite{Gor96}, C.~Gordon obtained some
structure results on GO-spaces, in particular, the following one: Every
Riemannian GO-manifold of nonpositive Ricci curvature is a
symmetric space.

\medskip
The goal of this paper is to clarify connections between Killing
fields of constant length on a Rimannian GO-manifold $(M,g)$ and
the structure of its full isometry group. The Lie algebra of the
full (connected) isometry group of  $(M,g)$  is identified
naturally with the Lie algebra of Killing fields $\mathfrak{g}$ on
$(M,g)$. We prove the following result: If $\mathfrak{a}$ is an
abelian ideal of $\mathfrak{g}$, then every Killing field $X\in
\mathfrak{a}$ has constant length (Theorem \ref{ideal}). On the ground
of this theorem we give a new proof of the above mentioned result
of C. Gordon on GO-manifold with nonpositive Ricci curvature
(Theorem \ref{gordonnonpos}).

\medskip

Let $(M, g)$ be a GO-manifolds and $G$ is its connected full
isometry group. Obviously, $(M, g)$ is homogeneous and $M=G/H$,
where $H$ is the isotropy subgroup at a point $o\in M$. Since $H$
is compact, there is an $\Ad(H)$-invariant decomposition

\begin{equation}\label{reductivedecomposition}
\mathfrak{g}=\mathfrak{h}\oplus \mathfrak{m},
\end{equation}
where $\mathfrak{g}={\rm Lie }(G)$ and $\mathfrak{h}={\rm Lie}
(H)$. The Riemannian metric $g$ is $G$-invariant and is determined
by an $\Ad(H)$-invariant Euclidean metric $g = (\cdot,\cdot)$ on
the space $\mathfrak{m}$ which is identified with the tangent
space $T_oM$ at the initial point $o = eH$.

In what follows we identify elements of $\mathfrak{g}$ with
corresponded Killing vector fields on $(M,g)$.

Now, we recall some well known formulas for a homogeneous
Riemannian manifold $(M,g=(\cdot, \cdot))$ \cite{Bes}. Let us
choose some $g$-orthonormal basis $(X_i)$ in $\mathfrak{m}$.
Consider also a vector $Z\in \mathfrak{m}$ defined by the
condition $(Z,X)=\trace (\ad_X)$ for every $X\in \mathfrak{m}$.
Therefore, $Z=0$ iff the Lie algebra $\mathfrak{g}$ (and the Lie
group $G$) is unimodular. The following formulae (that is more
simple for the unimodular case) is useful for the Ricci curvature
calculations:
$$
\Ric (X,X)=-\frac{1}{2}B_{\mathfrak{g}}(X,X)
-\frac{1}{2}\sum\limits_i |[X,X_i]_{\mathfrak{m}}|^2+
$$
\begin{equation}\label{ricc}
+\frac{1}{4}\sum\limits_{i,j}([X_i,X_j]_{\mathfrak{m}},X)^2-
([Z,X]_{\mathfrak{m}},X),
\end{equation}
where $B_{\mathfrak{g}}$ is the Killing form of the Lie algebra $\mathfrak{g}$,
$X\in \mathfrak{m}$, and
$V_{\mathfrak{m}}$ means the
$\mathfrak{m}$-part of a vector $V \in \mathfrak{g}$.

\begin{lemma}[\cite{KV}]\label{GO-criterion}
A homogeneous Riemannian manifold   $(M=G/H,g)$ with the reductive
decomposition  (\ref{reductivedecomposition}) is GO-space if and
only if  for any $X \in \mathfrak{m}$ there is $H_X \in
\mathfrak{h}$ such that

$([H_X +X,Y]_{\mathfrak{m}},X) =0$ for all $Y\in \mathfrak{m}$.
\end{lemma}

This lemma  shows that  the property to be GO-space  depends only
on the reductive decomposition (\ref{reductivedecomposition})  and
the Euclidean metric $g$ on $\mathfrak{m}$. In other words, if $(M
= G/H, g)$ is a GO-space, then any locally isomorphic homogeneous
Riemannian manifold $(M'= G'/H', g')$ is a GO-space.  Also  a
direct product of Riemannian manifolds is a manifold with
homogeneous geodesics if and only if each factor is a manifold
with homogeneous geodesics.

For any subspace $\mathfrak{l} \subset \mathfrak{g}$ and any $U\in
\mathfrak{g}$ we use the symbol $\ad_U^{\mathfrak{l}}$ for a
restriction of the operator $\ad_U$ to $\mathfrak{l}$, i.e.
$\ad_U^{\mathfrak{l}}: {\mathfrak{l}} \rightarrow {\mathfrak{l}}$,
$\ad_U^{\mathfrak{l}}(X)=[U,X]_{\mathfrak{l}}$.

\begin{lemma}\label{simple}
Suppose that $(M=G/H,g)$ is a GO-space. Let $\mathfrak{m}_1$ and $\mathfrak{m}_2$
be $\Ad(H)$-invariant subspaces of $\mathfrak{m}$
such that $(\mathfrak{m}_1,\mathfrak{m}_2)=0$ and $\mathfrak{m}=\mathfrak{m}_1 \oplus\mathfrak{m}_2$.
Then for any $U\in
\mathfrak{m}_1$ the operator  $\ad_U^{\mathfrak{m}_2}$ is
skew-symmetric. If, in addition, $[\mathfrak{h},\mathfrak{m}_1]=0$, then the operator $\ad_U^{\mathfrak{m}}$ is
skew-symmetric.
\end{lemma}

\begin{proof}
For any $X\in \mathfrak{m}_2$ there is $H_X \in \mathfrak{h}$ such
that $([H_X+X,Y]_{\mathfrak{m}},X) =0$ for all $Y\in \mathfrak{m}$
(see Lemma \ref{GO-criterion}). Therefore,
$$
0=([H_X+X,U]_{\mathfrak{m}},X)=([H_X,U]_{\mathfrak{m}},X)+([X,U]_{\mathfrak{m}},X)=([X,U]_{\mathfrak{m}},X),
$$
because $[H_X,U]\in \mathfrak{m}_1$ and $X\in \mathfrak{m}_2$. If $[\mathfrak{h},\mathfrak{m}_1]=0$, then
the same is true for any $X\in \mathfrak{m}$. This
proves the lemma.
\end{proof}

\begin{lemma}[\cite{Gor96}]\label{unimod}
Let $(M=G/H,g)$ be a GO-space, then the group $G$ is unimodular.
\end{lemma}

\begin{proof} Here we give a more direct proof, than the original one in \cite{Gor96}.
Suppose that the Lie algebra $\mathfrak{g}={\rm Lie} (G)$ is not unimodular and
consider its proper subspace
$$
\mathfrak{u}=\{X\in \mathfrak{g}\,|\, \trace(\ad_X)=0 \}.
$$
Obviously, $\mathfrak{h}\subset \mathfrak{u}$.
Since $\ad_{[X,Y]}=[\ad_X,\ad_Y]$ and $\trace(\ad_{[X,Y]})=\trace([\ad_X,\ad_Y])=0$, then
$[\mathfrak{u},\mathfrak{g}]\subset [\mathfrak{g},\mathfrak{g}]\subset \mathfrak{u}$, hence, $\mathfrak{u}$
is an ideal of $\mathfrak{g}$. Consider $\mathfrak{m}_1=\mathfrak{m}\cap \mathfrak{u}$ and let $\mathfrak{m}_2$
be a (non-trivial) $g$-orthogonal complement to $\mathfrak{m}_1$ in $\mathfrak{m}$.
Since $\mathfrak{u}$ is an ideal of $\mathfrak{g}$ and $g$ is $\ad(\mathfrak{h})$-invariant,
then $\mathfrak{m}_1$ and $\mathfrak{m}_2$ are
$\ad(\mathfrak{h})$-invariant. On the other hand,
$[\mathfrak{h},\mathfrak{m}_2]\subset [\mathfrak{g},\mathfrak{g}]\subset \mathfrak{u}$. Therefore,
$[\mathfrak{h},\mathfrak{m}_2]=0$.
Now, consider any non-trivial $Y\in \mathfrak{m}_2$. By our construction, $\trace (\ad_Y)\neq 0$.
On the other hand, by Lemma \ref{GO-criterion} for every $X\in \mathfrak{m}$
there is
$H_X \in \mathfrak{h}$ such
that $([H_X +X,Y]_{\mathfrak{m}},X) =0$. Since $[\mathfrak{h},\mathfrak{m}_2]=0$, we get
$([X,Y]_{\mathfrak{m}},X) =0$ (see also Lemma \ref{simple}), that implies $\trace(\ad_Y)=0$, a contradiction.
\end{proof}

\begin{lemma}\label{geodkill}
Let $(M,g)$ be a Riemannian manifold, $X$ a Killing field on
$(M,g)$. Consider any point $x\in M$ such that $X(x)\neq 0$. Then
the integral curve of $X$ through the point $x$ is a geodesic if
and only if $x$ is a critical point of the function $y\in M\mapsto
g_y(X,X)$.
\end{lemma}

\begin{proof}
In fact, this is proved in
Proposition 5.7  of Chapter VI in \cite{KN}.
\end{proof}

\smallskip

We use the symbol $M_x$ for the tangent space of a manifold $M$ at a point $x\in M$.

\begin{lemma}\label{newcrit}
Let $(M,g)$ be a Riemannian manifold, $\mathfrak{g}$ its Lie algebra of Killing field. Then $(M,g)$
is a GO-manifold if and only if for any $x\in M$ and any $v\in M_x$ there is
$X\in \mathfrak{g}$ such that $X(x)=v$ and $x$ is a critical point of the function
$y\in M\mapsto g_y(X,X)$. If $(M,g)$ is homogeneous, then the latter condition is equivalent
to the following one: for any $Y\in \mathfrak{g}$ the equality $g_x([Y,X],X)=0$ holds.
\end{lemma}

\begin{proof}
By Lemma \ref{geodkill} an integral curve of $X$ through the point $x\in M$ is geodesic if and only if
$x$ is a critical point of the function $y\in M\mapsto g_y(X,X)$.
If $(M,g)$ is homogeneous, then it is equivalent to the condition
$$
Y\cdot g(X,X)|_x=2g_x([Y,X],X)=0
$$
for every $Y \in \mathfrak{g}$.
\end{proof}

In what follows we need the following

\begin{prop}[Theorem 2.10 in \cite{YaBo}]\label{Bo}
If a Killing vector field $X$ on a compact Riemannian manifold $M$
satisfies the condition $\Ric(X,X)\leq 0$, then $X$ is parallel on
$M$ and $\Ric(X,X)\equiv 0$.
\end{prop}

\begin{corollary}\label{Bo1}
If a compact homogeneous Rimannian manifold $(M,g)$ has nonpositive
Ricci curvature, then it is an Euclidean torus. In particular, its full
connected isometry group is abelian.
\end{corollary}

\section{Main results}

At first, we get the following remarkable result.

\begin{theorem}\label{ideal}
Let $(M,g)$ be a GO-manifold, $\mathfrak{g}$ is its Lie algebra of Killing fields.
Suppose that $\mathfrak{a}$ is an abelian ideal of $\mathfrak{g}$. Then any $X\in \mathfrak{a}$
has constant length on $(M,g)$.
\end{theorem}

\begin{proof}
Let $x$ be any point in $M$. We will prove that $x$ is a critical
point of the function $y\in M\mapsto g_y(X,X)$. Since $(M,g)$ is
homogeneous, then (by Lemma \ref{newcrit}) it suffices to prove
that $g_x([Y,X],X)=0$ for every $Y \in \mathfrak{g}$.

Consider any $Y\in \mathfrak{a}$, then $Y\cdot g(X,X)=2g([Y,X],X)=0$ on $M$, since $\mathfrak{a}$ is abelian.

Now, consider $Y\in \mathfrak{g}$ such that $g_x(Y,U)=0$ for every
$U \in \mathfrak{a}$. We will prove that $g_x([Y,X],X)=0$. By
Lemma \ref{newcrit}, for the vector $X(x)\in M_x$ there is a
Killing field $Z\in \mathfrak{g}$ such that $Z(x)=X(x)$ and
$g_x([V,Z],Z)=0$ for any $V\in \mathfrak{g}$. In particular,
$g_x([Y,Z],Z)=0$. Now, $W=X-Z$ vanishes at $x$ and we get
$$
g_x([Y,X],X)=g_x([Y,Z+W],Z+W)=g_x([Y,Z+W],Z)=
$$
$$
g_x([Y,Z],Z)+g_x([Y,W],Z)=g_x([Y,W],Z).
$$
Note that $g_x([Y,W],Z)=-g_x([W,Y],Z)=g_x(Y,[W,Z])=0$ because
$W(x)=0$ ($0=W\cdot g(Y,Z)|_x=g_x([W,Y],Z)+g_x(Y,[W,Z])$) and
$[W,Z]=[X,Z] \in \mathfrak{a}$. Therefore, $g_x([Y,X],X)=0$.
Hence, $x$ is a critical point of the function $y\in M\mapsto
g_y(X,X)$.

Since every $x\in M$ is a critical
point of the function $y\in M\mapsto g_y(X,X)$, then
$X$ has constant length on $(M,g)$.
\end{proof}

\begin{remark}\label{rem1}
This result can be easily generalized to some cases when
$\mathfrak{g}$ is a subalgebra of the Lie algebra of the full
connected isometry group of $(M,g)$. It suffices that a connected
subgroup $G$ (with the Lie algebra $\mathfrak{g}$) of the full
isometry group of $(M,g)$ is such that $(M=G/H,g)$ is a GO-space.
\end{remark}

In the rest of the paper
we reprove the following

\begin{theorem}[C.~Gordon \cite{Gor96}]\label{gordonnonpos}
Every Riemannian GO-manifold of nonpositive Ricci curvature is symmetric.
\end{theorem}

\begin{remark}
It should be noted that the original proof of this theorem
(Theorem 5.1 in \cite{Gor96}) has an error in the claim ``Since
$U^*/L^*$ is a compact homogeneous space, its Ricci curvature
$\Ric^*$ is nonnegative''. Nevertheless, this error could be
corrected, and the proof in \cite{Gor96} requires only a little
modification. But here we present a more simple proof, in which
some constructions from \cite{Gor96} are essentially used.
\end{remark}

It suffices to prove Theorem \ref{gordonnonpos} for {\bf simply
connected manifolds}. Indeed, if a Rimannian homogeneous manifold
$M$ has a Riemannian symmetric space of nonpositive Ricci
curvature (equivalently, nonpositive sectional curvature) as a
universal covering, then it is symmetric too \cite{W,Hel}.

Let $(M, g)$ be a simply connected GO-manifolds with nonpositive Ricci curvature,
an let $G$ be its full connected isometry group. We know that $G$ is unimodular (Lemma \ref{unimod}),
and the isotropy subgroup $H$ must be connected. At first, we reduce the problem to the case when
$G$ is semisimple.

\begin{prop}\label{nonposricstr}
Let $(M,g)$ be a simply connected GO-manifold with nonpositive
Ricci curvature. Then it is a direct metric product of an
Euclidean space $\mathbb{E}^m$ and a simply connected GO-manifold
$(M_1,g_1)$ (with nonpositive Ricci curvature) with semisimple
full isometry groups.
\end{prop}

\begin{proof}
In the notation of Theorem \ref{ideal}, any Killing field $X\in
\mathfrak{a}$ has constant length on $(M,g)$. Since the Ricci
curvature is nonpositive, then by Theorem 4 in \cite{BerNikKF} we
get that $\Ric(X,X)=0$, moreover, the Killing field $X$ is
parallel on $(M,g)$, and the Riemannian manifold $(M,g)$ is a
direct metric product of two Riemannian manifolds, one of which is
a one-dimensional manifold tangent to Killing field $X$ and
another one is a GO-manifold with nonpositive Ricci curvature.
This procedure could be repeated unless the obtained second
manifold has a semisimple full isometry group.
\end{proof}

{\bf In what follows we suppose that the group $G$ is semisimple}. Now
we consider a reductive decomposition (see
(\ref{reductivedecomposition}))
$$
\mathfrak{g}=\mathfrak{h}\oplus \mathfrak{m},
$$
where $\mathfrak{g}={\rm Lie }(G)$, $\mathfrak{h}={\rm Lie} (H)$,
and $\mathfrak{m}$ is an orthogonal complement to $\mathfrak{h}$
in $\mathfrak{h}$ with respect to the Killing form
$B_{\mathfrak{g}}$ of the (semisimple) Lie algebra $\mathfrak{g}$.
The Riemannian metric $g$ is $G$-invariant and is determined by an
$\Ad(H)$-invariant Euclidean metric $g = (\cdot,\cdot)$ on  the
space $\mathfrak{m}$. Now we consider a {\bf maximal compactly
embedded} subalgebra $\mathfrak{k}\subset \mathfrak{g}$ such that
$\mathfrak{h}\subset \mathfrak{k}$.

If $\mathfrak{h}=\mathfrak{k}$, then the manifold
under consideration is a symmetric space \cite{W,Hel}.
Suppose now, that
$\mathfrak{h}\neq \mathfrak{k}$. Then there are $\Ad(H)$-invariant
subspaces $\mathfrak{m}_1, \mathfrak{m}_2 \subset \mathfrak{m}$
such that $(\mathfrak{m}_1,\mathfrak{m}_2)=0$,
$\mathfrak{m}=\mathfrak{m}_1 \oplus\mathfrak{m}_2$, and
$\mathfrak{k}=\mathfrak{h}\oplus \mathfrak{m}_1$.

Let $K^*$ be a compact Lie group with the Lie algebra
$\mathfrak{k}$ and $H^*$ be its subgroup corresponded to
subalgebra $\mathfrak{h}\subset \mathfrak{k}$. We have
$\mathfrak{k}=\mathfrak{h}\oplus \mathfrak{m}_1$, therefore
$\mathfrak{m}_1$ could be identified with the tangent space at the
point $eH^*$ of a compact homogeneous manifold $M^*=K^*/H^*$. We
consider $K^*$-invariant Riemannian metric $g^*$ on $M^*$ that is
generated with the inner product $(\cdot,
\cdot)|_{\mathfrak{m}_1}$. Note that $K^*$ may not act effectively
on $M^*=K^*/H^*$, but this is not important for calculation of the
Ricci curvature of $(M^*, g^*)$.

We choose a $(\cdot, \cdot)$-orthonormal basis $X_1,X_2, \dots,
X_r$, $r=\dim(\mathfrak{m}_1)$, in $\mathfrak{m}_1$, and $(\cdot,
\cdot)$-orthonormal basis $Y_1,Y_2, \dots, Y_s$,
$s=\dim(\mathfrak{m}_2)$, in $\mathfrak{m}_2$. Denote the Ricci
curvature of $(M^*, g^*)$ by $\Ric^*$ and the Killing forms of
$\mathfrak{k}$ and $\mathfrak{g}$ by $B_{\mathfrak{k}}$ and
$B_{\mathfrak{g}}$ respectively.
Using (\ref{ricc}) and the fact that
$\mathfrak{g}$ is unimodular
(Lemma \ref{unimod} is not necessary, because every semisimple Lie algebra is unimodular),
we get
$$
\Ric (X,X)=-
\frac{1}{2}B_{\mathfrak{g}}(X,X)-\frac{1}{2}\sum\limits_i
|[X,X_i]_{\mathfrak{m}}|^2-\frac{1}{2}\sum\limits_i
|[X,Y_i]_{\mathfrak{m}}|^2+
$$
$$
+\frac{1}{4}\sum\limits_{i,j}([X_i,X_j]_{\mathfrak{m}},X)^2
+\frac{1}{4}\sum\limits_{i,j}([Y_i,Y_j]_{\mathfrak{m}},X)^2
+\frac{1}{2}\sum\limits_{i,j}([X_i,Y_j]_{\mathfrak{m}},X)^2
$$
and
$$
\Ric^* (X,X)=-
\frac{1}{2}B_{\mathfrak{k}}(X,X)-\frac{1}{2}\sum\limits_i
|[X,X_i]_{\mathfrak{m}_1}|^2+
\frac{1}{4}\sum\limits_{i,j}([X_i,X_j]_{\mathfrak{m}_1},X)^2
$$
for any $X\in \mathfrak{m}_1$.
By Lemma \ref{simple}
$$
B_{\mathfrak{g}}(X,X)=B_{\mathfrak{k}}(X,X)+\sum\limits_{i}([X,[X,Y_i]]_{\mathfrak{m}},Y_i)=
B_{\mathfrak{k}}(X,X)-\sum\limits_{i}|[X,Y_i]_{\mathfrak{m}_2}|^2,
$$
then, using  Lemma \ref{simple} again
($([X_i,Y_j]_{\mathfrak{m}},X)^2=([Y_j,X_i]_{\mathfrak{m}_1},X)^2=([Y_j,X]_{\mathfrak{m}_1},X_i)^2$
and $\sum\limits_{i,j}([X_i,Y_j]_{\mathfrak{m}},X)^2=\sum\limits_i
|[X,Y_i]_{\mathfrak{m}_1}|^2$), we get

\begin{prop}[\cite{Gor96}]\label{eneq}
For any $X\in \mathfrak{m}_1$ the equality
$$
\Ric^* (X,X)=\Ric (X,X)-\frac{1}{2}\sum\limits_{1 \leq i <j \leq
r}([Y_i,Y_j]_{\mathfrak{m}_1},X)^2
$$
holds.
\end{prop}

Since $(M=G/H,g)$ has nonpositive Ricci curvature, then
from Proposition \ref{eneq} we get $\Ric^* (X,X) \leq 0$ for any
$X\in \mathfrak{m}_1$, but $\Ric^* (X,X) = 0$ implies $\Ric(X,X) = 0$
and $([\mathfrak{m}_2,\mathfrak{m}_2]_{\mathfrak{m}_1},X)=0$.

Since $(M^*=K^*/H^*,g^*)$ is a compact homogeneous Riemannian
manifold with nonpositive Ricci curvature, then it is Euclidean torus by
Corollary \ref{Bo1}.
This implies $\Ric^* (X,X) = \Ric (X,X)= 0$ for all $X\in
\mathfrak{m}_1$, $\mathfrak{m}_1$ lies in the center of
$\mathfrak{k}$, and $[\mathfrak{m}_2,\mathfrak{m}_2]\subset
\mathfrak{h}\oplus \mathfrak{m}_2$.

Let $\mathfrak{p}$ be a $B_{\mathfrak{g}}$-orthogonal compliment
to $\mathfrak{k}$ in $\mathfrak{g}$. It is well known that
$[\mathfrak{p},\mathfrak{p}]\subset \mathfrak{k}$; if
$\mathfrak{h}_1:=[\mathfrak{p},\mathfrak{p}]$, then
$\mathfrak{g}_1:=\mathfrak{h}_1\oplus \mathfrak{p}$ is a maximal
noncompact semisimple ideal in the Lie algebra $\mathfrak{g}$.

Now we will prove that $\mathfrak{p}=\mathfrak{m}_2$. By Lemma
\ref{simple} we get that for any $U\in \mathfrak{m}_1$ the
operator $\ad_U^{\mathfrak{m}}$ is skew-symmetric. The same is
true for any $U\in \mathfrak{h}$, and, therefore, for any $U\in
\mathfrak{k}=\mathfrak{h}\oplus \mathfrak{m}_1$. From the
relations $\mathfrak{p}\subset \mathfrak{m}$,
$\mathfrak{h}_1\subset \mathfrak{k}$,
$[\mathfrak{k},\mathfrak{m}_1]=0$, and
$[\mathfrak{k},\mathfrak{p}]=\mathfrak{p}$ we get
$$
(\mathfrak{m}_1,\mathfrak{p})=(\mathfrak{m}_1,[\mathfrak{k},\mathfrak{p}]_{\mathfrak{m}})=
([\mathfrak{k},\mathfrak{m}_1]_{\mathfrak{m}},\mathfrak{p})+
(\mathfrak{m}_1,[\mathfrak{k},\mathfrak{p}]_{\mathfrak{m}})=0,
$$
since the operators $\ad_U^{\mathfrak{m}}$ are skew-symmetric  for
all $U\in {\mathfrak{k}}$. This proves
$\mathfrak{p}=\mathfrak{m}_2$. Therefore,
$\mathfrak{h}_1:=[\mathfrak{p},\mathfrak{p}]=[\mathfrak{m}_2,\mathfrak{m}_2]\subset
\mathfrak{h}\oplus \mathfrak{m}_2$, and we get
$\mathfrak{h}_1\subset \mathfrak{h}$.

Let $\mathfrak{g}_2$ be a $B_{\mathfrak{g}}$-orthogonal compliment
to $\mathfrak{g}_1$ in $\mathfrak{g}$. Then $\mathfrak{g}_2$ is a
maximal compact semisimple ideal in the Lie algebra
$\mathfrak{g}$. If $\mathfrak{h}_2$ is a
$B_{\mathfrak{g}}$-orthogonal compliment to $\mathfrak{h}_1$ in
$\mathfrak{h}$, then $\mathfrak{g}_2=\mathfrak{h}_2\oplus
\mathfrak{m}_1$. Recall that
$[\mathfrak{h}_2,\mathfrak{m}_1]\subset
[\mathfrak{h},\mathfrak{m}_1]=0$. Hence, $\mathfrak{h}_2$ is an
ideal in the Lie algebra $\mathfrak{g}$. Since the space $G/H$ is
effective, $\mathfrak{h}_2$ is trivial. Since
$\mathfrak{m}_1=\mathfrak{g}_2$ is commutative ($\mathfrak{m}_1$ lies in the center of
$\mathfrak{k}$), then it is trivial too (otherwise, $\mathfrak{g}_2$ is not semisimple). Hence,
$\mathfrak{h}=\mathfrak{k}$, and $(M=G/H, g)$ is a symmetric
space.

Therefore, Theorem \ref{gordonnonpos} is completely proved.

\medskip

{\bf Acknowledgements.}  The project was supported in part by the
State Maintenance Program for the Leading Scientific Schools of
the Russian Federation (grant NSh-6613.2010.1), by Federal Target
Grant ``Scientific and educational personnel of innova\-tive
Russia'' for 2009-2013 (government contract No. 02.740.11.0457),
and by RFBR (grant 10-01-9000-Bel-a).

\end{document}